\documentclass[10pt]{amsart}
\usepackage{amssymb,latexsym,amsmath,amsthm}
\usepackage[mathscr]{euscript}\textwidth 4.7in \textheight 7.5in

\begin{document}
%%%%%%%%%%%%%%%%%%%%%%%%%%%%%%%%%%%%%%%%%%%%%%%%%%%%%%%%%%%%%%%%%%%%%%%
%%%%%%%%%%%%%%%%%%%%%%%%%     Macros      %%%%%%%%%%%%%%%%%%%%%%%%%%%%%
%%%%%%%%%%%%%%%%%%%%%%%%%%%%%%%%%%%%%%%%%%%%%%%%%%%%%%%%%%%%%%%%%%%%%%%
\numberwithin{equation}{section}

\newtheorem{lem}[equation]{Lemma}
\newtheorem{cor}[equation]{Corollary}
\newtheorem{thm}[equation]{Theorem}
\newtheorem{rem}[equation]{Remark}
\newtheorem*{thm*}{Theorem}
% shorthand
\def\G2{{G_2}}
\def\ol{\overline}
\def\op{\oplus}
\def\ot{\otimes}
\def\del{\partial}
\def\CK{Cartan-K\"ahler }
\def\inv{{}^{-1}}
% blackboard style
\def\bC{\mathbb{C}}
\def\bO{\mathbb{O}}
\def\bP{\mathbb{P}}
\def\bR{\mathbb{R}}
% mathcal style
\def\cF{\mathcal{F}}
\def\cI{\mathcal{I}}
\def\cL{\mathcal{L}}
\def\cP{\mathcal{P}}
\def\cV{\mathcal{V}}
% mathfrak
\def\fg{\mathfrak{g}}
\def\fgl{\mathfrak{gl}}
\def\fh{\mathfrak{h}}
\def\fso{\mathfrak{so}}
\def\fspin{\mathfrak{spin}}
% mathscript
\def\sI{\mathscr{I}}
% text style
\def\tad{\mathrm{ad}}
\def\td{\mathrm{d}}
\def\dx{\mathrm{d}x}
\def\dy{\mathrm{d}y}
\def\dV{\mathrm{d}V}
\def\dy{\mathrm{d}y}
\def\tdim{\mathit{dim}}
\def\tcodim{\mathit{codim}}
\def\tGL{\mathrm{GL}}
\def\tGr{\mathrm{Gr}}
\def\tHol{\mathrm{Hol}}
\def\bi{\mathbf{i}}
\def\tker{\mathit{ker}}
\def\tId{\mathrm{Id}}
\def\tIm{\mathrm{Im}}
\def\tmod{ \ \mathrm{mod} \ }
\def\tSL{\mathrm{SL}}
\def\tSO{\mathrm{SO}}
\def\tSU{\mathrm{SU}}
\def\tspan{\mathrm{span}}
\def\tSpin{\mathit{Spin}}
\def\tvol{\mathrm{d}\mathit{vol}}
% greek
\def\a{\alpha}
\def\w{\omega}
% fractions
\def\half{\textstyle{\frac12}}

%%%%%%%%%%%%%%%%%%%%%%%%%%%%%%%%%%%%%%%%%%%%%%%%%%%%%%%%%%%%%%%%%%%%%%%
%%%%%%%%%%%%%%%%%%%%%     Text of paper    %%%%%%%%%%%%%%%%%%%%%%%%%%%%
%%%%%%%%%%%%%%%%%%%%%%%%%%%%%%%%%%%%%%%%%%%%%%%%%%%%%%%%%%%%%%%%%%%%%%%

\title[Associative and Cayley Embeddings]
{Calibrated associative and Cayley embeddings}
\author[Colleen Robles and Sema Salur]{Colleen Robles and Sema Salur}
\address{Department of Mathematics, Texas A$\&$M University, College Station, TX, }
\address {Department of Mathematics, University of Rochester, Rochester, NY, 14627}
\email{salur@math.rochester.edu } 
\date{\today}

\begin{abstract}
Using the Cartan-K\"ahler theory, and results on real algebraic
structures, we prove two embedding theorems.  First, the interior
of a smooth, compact 3-manifold may be isometrically embedded into
a $G_2$-manifold as an associative submanifold.  Second, the
interior of a smooth, compact 4-manifold $K$, whose double
$\mathit{doub}(K)$ has a trivial bundle of self-dual 2-forms, may
be isometrically embedded into a $\tSpin(7)$-manifold as a Cayley
submanifold. Along the way, we also show that Bochner's Theorem on
real analytic approximation of smooth differential forms, can be
obtained using real algebraic tools developed by Akbulut and King.
\end{abstract}
\maketitle

%%%%%%%%%%%%%%%%%%%%%%%%%%%%%%%%%%%%%%%%%%%%%%%%%%%%%%%%%%%%%%%%%%%%%
%%%%%%%%%%%%%%%%%%%%%%%%%%%%%%%%%%%%%%%%%%%%%%%%%%%%%%%%%%%%%%%%%%%%%

\section{Introduction}\label{sec:intro}

Let $(M^7,g)$ be a Riemannian 7-manifold whose holonomy group
$\tHol(g)$ is a subgroup of the exceptional group $G_2$.
Then $M$ is naturally equipped with a covariantly constant 3-form $\varphi$
and 4-form $*\varphi$. We call $(M,\varphi,g)$ a {\it $G_2$-manifold}.
It is well known that $\varphi$ and $*\varphi$ are {\it
calibrations} on $M$, in the sense of Harvey and Lawson \cite{HL}.
The corresponding calibrated submanifolds in $M$ are called {\it
associative $3$-folds} and {\it coassociative $4$-folds},
respectively.

\vspace{.1in}

Similarly, if $(M^8,g)$ has $\tHol(g) \subseteq \tSpin(7)$, then
$M$ admits a covariantly constant, self-dual 4-form $\Psi$, and we
call $(M,\Psi,g)$ a {\it $\tSpin(7)$-manifold}.  The 4-form $\Psi$
is the {\it Cayley calibration}, and the calibrated submanifolds
are {\it Cayley 4-folds}.

\vspace{.1in}

Constructing examples of manifolds with $G_2$ and $Spin(7)$
holonomy and their calibrated submanifolds is of interest because
of their importance in string theory. Also, they provide new
examples of volume minimizing submanifolds in a given homology
class \cite{HL}. In \cite{BrClbEm}, R. Bryant applied the
Cartan-K\"{a}hler theory to show that: (1) every closed, real
analytic, oriented Riemannian 3-fold can be isometrically embedded
in a Calabi-Yau 3-fold as a special Lagrangian submanifold; and
(2) every closed, real analytic, oriented Riemannian 4-fold with a
trivial bundle of self-dual 2-forms can be isometrically embedded
in a $G_2$-manifold as an coassociative submanifold.  Moreover,
the submanifolds above may be embedded as the fixed locus of a
real structure (in the special Lagrangian case), or an anti
$G_2$-involution (in the coassociative case).

\vspace{.1in}

In this paper, we will first show that Bryant's constructions can
be repeated for the associative and Cayley submanifolds.

%------------------------------------------------------------------
\begin{thm}\label{thm:assoc}
%------------------------------------------------------------------
Assume $(K^3,g)$ is a closed, oriented, real analytic Riemannian
3-manifold. Then there exists a $\G2$-manifold $(N^7,\varphi)$ and
an isometric embedding $i : K \hookrightarrow N$ such that the
image $i(K)$ is an associative submanifold of $N$. Moreover,
$(N,\varphi)$ can be chosen so that $i(K)$ is the fixed point set
of a nontrivial $G_2$-involution $r : N \to N$.
\end{thm}
%------------------------------------------------------------------

\noindent{\it Remark.}  Bryant showed that $K$ isometrically embeds as a special Lagrangian submanifold of a Calabi-Yau 3-fold $CY$.  This immediately yields an elementary version of Theorem \ref{thm:assoc} as $N = CY \times \bR$ naturally carries a $G_2$ structure such that $A$ isometrically embeds as an associative submanifold.  However, $CY \times \bR$ has holonomy a subgroup of $\tSU(3)$.  It may be checked that, as long as $A$ is not flat, the $N$ of Theorem \ref{thm:assoc} has holonomy exactly $G_2$.  In particular, these $N$ are not of the form $CY \times \bR$.  See the remark at the end of \S\ref{sec:assocproof}.
\begin{thm}\label{thm:cayley}
%------------------------------------------------------------------
Assume $(K^4,g)$ is a closed, oriented, real analytic Riemannian
4-manifold with a trivial bundle of self-dual 2-forms. Then there
exists a $\tSpin(7)$-manifold $(N^8,\Psi)$ and an isometric
embedding $i : K \hookrightarrow N$ whose image is a Cayley
submanifold in $N$. Moreover, $(N,\Psi)$ can be chosen so that
$i(K)$ is the fixed locus of a nontrivial $\tSpin(7)$-involution
$r : N \to N$.
\end{thm}
%------------------------------------------------------------------

We refer the reader to \cite[\S0.4]{BrClbEm} for a discussion of Cartan-K\"ahler theory that will be used in the constructions.

\vspace{.1in}

Making use of the real analytic implicit function theorem and a
theorem of Nash-Tognoli we are able to show that Theorems
\ref{thm:assoc} and \ref{thm:cayley} extend to interiors of
compact, smooth manifolds.  In particular, assume that $K$ is a
compact, oriented, smooth manifold, possibly with boundary. Let
$\mathit{doub}(K)$ denote the {\it doubling of $K$}: glue two
copies of $K$ together along the boundary with the identity map.
If $K$ is closed ($\partial M = \emptyset$) then $\mathit{doub}(K)
= K$.  The manifold $\mathit{doub}(K)$ is closed and orientable, and admits
the structure of a real analytic Riemannian manifold; see Lemma
\ref{lem:analyticstr}.  Then we have the following two
corollaries.

%------------------------------------------------------------------
\begin{thm}\label{thm:openassoc}
%------------------------------------------------------------------
Let $A$ be the interior of a smooth, orientable, compact
3-manifold $K$ with nonempty boundary. Then $A$ admits a
compatible real analytic Riemannian structure. There exists a
$\G2$-manifold $(N^7,\varphi)$ and an isometric embedding $i : A
\hookrightarrow N$ such that $i(A)$ is an associative submanifold
in $N$. Moreover $(N,\varphi)$ may be chosen so that $i(A)$ is the
fixed locus of a nontrivial $G_2$-involution $r: N \to N$.
\end{thm}

%------------------------------------------------------------------
\begin{thm}\label{thm:opencayley}
%------------------------------------------------------------------
Let $A$ be the interior of a smooth, orientable, compact
4-manifold $K$ with nonempty boundary. Then $A$ admits a
compatible real analytic Riemannian structure. Assume also that
the bundle of self-dual 2-forms over $doub(K)$ is trivial. There
exists a $\tSpin(7)$-manifold $(N^8,\Psi)$ and an isometric
embedding $i : A \hookrightarrow N$ whose image $i(A)$ is a Cayley
submanifold in $N$. Moreover, $(N,\Psi)$ may be chosen so that
$i(A)$ is the fixed point set of a nontrivial
$\tSpin(7)$-involution $r: N \to N$.
\end{thm}
%------------------------------------------------------------------
\noindent Theorems \ref{thm:assoc} \& \ref{thm:openassoc} and
Theorems \ref{thm:cayley} \& \ref{thm:opencayley} are proven in
\S\ref{sec:assocproof} and \ref{sec:cayleyproof}, respectively.
Also note that in all these theorems $N$ does not have to be a
(locally) product manifold.

\medskip

\noindent {\it Acknowledgments.}  The authors thank Robert Bryant
and Selman Akbulut for illuminating discussions. This paper was
strongly influenced by Bryant's \cite{BrClbEm}.

%-------------------------------------------------------------------------
\section{Associative submanifolds of $\G2$-manifolds}
%-------------------------------------------------------------------------
\subsection{\boldmath $\G2$-manifolds and the associative
calibration \boldmath}\label{sec:G2}
%-------------------------------------------------------------------------

On the imaginary octonians $\bR^7 = \tIm(\bO)$ let $x = (x^j)$
denote the standard linear coordinates, set $\dx^{jk} :=
\dx^j\wedge\dx^k$, and define the 3-forms $\dx^{jk\ell} := \dx^j
\wedge \dx^k \wedge \dx^\ell$ and
\begin{displaymath}
  \varphi_0 := \dx^{123} + \dx^1 \wedge \left( \dx^{45} + \dx^{67} \right)
          + \dx^2 \wedge \left( \dx^{46} - \dx^{57} \right)
          + \dx^3 \wedge \left( - \dx^{47} - \dx^{56} \right) \, .
\end{displaymath}
The simple Lie group $\G2$ is the subgroup of $\tGL(7)$ preserving
$\varphi_0$ \cite{BrExHol}.

\vspace{.1in}

A {\it $\G2$-structure} on $M^7$ is a principle right $\G2$-bundle $\pi:
P \to M$.  The elements of $P_x = \pi^{-1}(x)$ are linear
isomorphisms $u : T_x M \to \bR^7$, and the right action is given
by $u \cdot a = a^{-1} \circ u$.  The $\G2$-structure induces a
well-defined 3-form $\varphi$ on $M$ via $\varphi_x =
u^*\varphi_0$. Additionally, $M$ admits a unique metric $g$ and
volume form $\ast 1$ (also obtained by pull-back) for which $u :
T_x M \to \bR^7$ is an oriented isometry. In particular,
$(\ast\varphi)_x = u^*(\ast\varphi_0)$.

\vspace{.1in}

We say that $(M,\varphi)$ is a {\it $\G2$-manifold} when $\varphi$
and $\ast\varphi$ are closed. Equivalently, the $\G2$-structure is
torsion-free \cite{FG}. In this case, $\varphi$ is parallel, $M$
is Ricci-flat \cite[10.64]{Be}, and the metric is real analytic
in harmonic coordinates \cite[Th. 5.2]{DK}.  Since $\varphi$ is harmonic it follows that $\varphi$ is real analytic as well.

\vspace{.1in}

Assume $(M,\varphi)$ is a $\G2$-manifold. Then
$\varphi$ is the {\it associative calibration}.  The 3-dimensional
submanifolds $i : X^3 \hookrightarrow M$ calibrated by $\varphi$
are the {\it associative submanifolds}. Associative submanifolds
are plentiful in $\G2$-manifolds: it is a consequence of the Cartan-K\"ahler theorem \cite[\S0.4]{BrClbEm} that every associative $E^3 \subset T_zM$ is
tangent to an associative $X^3 \subset M$. Moreover,

%--- Lemma: Assoc---------------------------------------------------------
\begin{lem}
\label{lem:assoc} Every real-analytic 2-dimensional submanifold
$Y^2$ of a $\G2$-manifold $(M^7,\varphi)$ lies in a unique
associative $X^3$.
\end{lem}
%-------------------------------------------------------------------------
\noindent The flat case $(M,\varphi) = (\bR^7,\varphi_0)$ was
proven by Harvey and Lawson \cite[Th.4.1]{HL}.  Given Lemma
\ref{lem:closed-assoc-ideal} below, the proof (at the end of this section) is a simple application of the Cartan-K\"ahler theorem \cite[\S0.4]{BrClbEm}.

\vspace{.1in}

The fundamental identity \cite[Th.1.6]{HL}
\begin{displaymath}
  \varphi(u,v,w)^2 + | \chi(u,v,w) |^2 \ = \ | u\wedge v \wedge w|^2
\end{displaymath}
implies that $i^*\varphi = \textrm{d} \textit{vol}$ precisely when
$i^*\chi = 0$. Here $\chi$ is the vector-valued 3-form defined by
\begin{displaymath}
  \langle \chi(u,v,w) , z \rangle = \ast\varphi(u,v,w,z) \, .
\end{displaymath}
In particular, the associative submanifolds are the 3-dimensional
integral manifolds of $\{\chi = 0\}$. In the flat case
$(\bR^7,\varphi_0)$,
\begin{displaymath}
  \ast\varphi_0 = \dx^{4567} + \dx^{23}\wedge\left( \dx^{45} + \dx^{67} \right)
  + \dx^{31} \wedge \left( \dx^{46} - \dx^{57} \right)
  + \dx^{12} \wedge \left(  - \dx^{47} - \dx^{56} \right) \, ,
\end{displaymath}
and
\begin{eqnarray*}
  \chi_0 & = & - \ \left( \dx^{357} - \dx^{346} - \dx^{256} - \dx^{247} \right)
               \del_{x^1} \\
         & & - \ \left( \dx^{367} + \dx^{345} + \dx^{156} + \dx^{147} \right)
               \del_{x^2} \\
         & &  + \ \left( \dx^{267} + \dx^{245} + \dx^{157} - \dx^{146} \right)
               \del_{x^3} \\
         & & - \ \left( \dx^{567} + \dx^{235} - \dx^{136} - \dx^{127} \right)
               \del_{x^4} \\
         & &  + \ \left( \dx^{467} + \dx^{234} - \dx^{137} + \dx^{126} \right)
               \del_{x^5} \\
         & & - \  \left( \dx^{457} + \dx^{237} + \dx^{134} + \dx^{125} \right)
               \del_{x^6} \\
         & &  + \ \left( \dx^{456} + \dx^{236} + \dx^{135} - \dx^{124} \right)
               \del_{x^7} \, .
\end{eqnarray*}
Notice that the coefficient 3-forms are $\chi_{0,j} := -\del_{x^j}
\, \lrcorner \, (\ast\varphi_0)$.

\vspace{.1in}

Given an arbitrary $\G2$ manifold $(M^7,\varphi)$, let $\{ \w^1 \,
, \, \ldots \, , \, \w^7 \}$ be a local $\G2$ coframing.  That is,
$\w^j_x = u_x^* \dx^j$ for smoothly varying isometries $u_x : T_x
M \to \bR^7$ in $P_x$.  Let $\{ e_j \}$ denote the dual framing.
Then local expressions for $\varphi = u^*\varphi_0$, $\ast\varphi
= u^*(\ast\varphi_0)$ and $\chi = u^*\chi_0$ are given by
replacing the terms $\dx$ and $\del_x$ in $\varphi_0$,
$\ast\varphi_0$ and $\chi_0$ with $\w$ and $e$, respectively.  The
associative submanifolds of $(M,\varphi)$ are the 3-dimensional
integral manifolds of $\{\chi_j := -e_j \, \lrcorner \,
(\ast\varphi)\}$.

%--- Lemma ---------------------------------------------------------------
\begin{lem}\label{lem:closed-assoc-ideal}
Let $\cI$ be the ideal algebraically generated by the coefficient
3-forms $\chi_j$.  Then $\cI$ is well-defined and closed under
exterior differentiation ($\td\cI \subset \cI$).
\end{lem}
%-------------------------------------------------------------------------
\begin{proof}
That $\cI$ is well-defined (i.e. does not depend on choice of
local $G_2$-coframing $u_x$) is immediate from the
$\G2$-invariance of the forms $\varphi$, $\ast\varphi$ and $\chi$.
To see that $\cI$ is differentially closed recollect that the Lie
derivative of any form $\a$ by a vector field $X$ is $\cL_X \a = X
\lrcorner \, d \a + d( X \lrcorner \, \a )$, so that
\begin{displaymath}\renewcommand{\arraystretch}{1.3}\begin{array}{rclr}
  d \chi_j & = & -d \left( e_j \, \lrcorner \, \ast\!\varphi \right) & \\
  & = & e_j \, \lrcorner \, d (\ast\varphi) \ - \ \cL_{e_j} \ast\!\varphi & \\
  & = & - \cL_{e_j} \ast\!\varphi \ ,
      & (\ast\varphi \textrm{ is closed})\, , \\
  & = & \left( e_j \, \lrcorner \, d\w^k \right) \wedge \chi_k \, .
      &
\end{array}\end{displaymath}
The last line follows from an application of \cite[\S{V}.8,
Ex.8]{Boothby}.
\end{proof}

\medskip

\noindent{\it Proof of Lemma \ref{lem:assoc}.} Since $Y$ is
2-dimensional and $\cI$ is generated by 3-forms, $Y$ is a priori
an integral manifold.  In order to apply the Cartan-K\"ahler theorem we must
show that: (i) $Y$ is regular; and (ii) the polar space $H(T_y Y)$
is of dimension $3$ for every $y \in Y$.  (See \cite[\S0.4]{BrClbEm} for a review of polar spaces, the variety of $p$-dimensional integral elements $V_p(\cI)_y$ in the Grassmannian $\tGr(p,T_yM)$ and the Cartan-K\"ahler theorem in this context.)

\vspace{.1in}

Regularity is easily confirmed, and in the course of doing so we
will see that the polar space is of dimension three. Fix $y\in Y$.
Since $\cI$ is generated by 3-forms, $V_2(\cI)_y = \tGr(2,T_yM)$,
and $V_2(\cI)$ is (trivially) a smooth submanifold of $\tGr(2,TM)$
near $T_y Y$.  Whence, $T_y Y$ is ordinary.

\vspace{.1in}

Because $\G2$ acts transitively on 2-planes, there is no loss of
generality in assuming that $T_y Y$ is spanned by $\{ e_1 , e_2
\}$.  The polar space of $T_y Y$ is $H(T_y Y) =
 \{ v \in T_y M \ | \ \psi(v,e_1,e_2) = 0 \ \forall \ \psi \in \cI^3 \}$.
In our case it is straightforward to see that $H(T_y Y) = \{ v \in
T_y M \ | \ \chi_j(v,e_1,e_2) = 0 \ \forall \ j \}$ is spanned by
$\{ e_1 , e_2 , e_3 \}$.  Whence the extension rank $r(T_yY) =
\tdim H(T_yY) - (2+1) = 0$ is constant function $V_2(\cI) =
\tGr(2,TM)$, and $Y$ is regular. The result follows from the Cartan-K\"ahler theory.\hfill\qed

%-------------------------------------------------------------------------
\subsection{\boldmath$G_2$-involutions\unboldmath}\label{sec:G2invol}
%-------------------------------------------------------------------------

One way of finding examples of associative submanifolds is to
investigate the fixed point sets of $G_2$ involutions,
\cite[Prop.10.8.1 ]{Joy1}.

\vspace{.1in}

Let $\sigma :M\rightarrow M$ be a nontrivial isometric involution
of a $G_2$-manifold $(M,g)$. This means that $\sigma:M\rightarrow
M$ is a diffeomorphism satisfying $\sigma^*(g)=g$ and
$\sigma^2=id$, but $\sigma\neq 1$.

\vspace{.1in}

\begin{lem}
\label{involution1}
Let $(M,\varphi,g)$ be  a $G_2$-manifold and let $\sigma:M\rightarrow M$ be a
nontrivial isometric involution preserving $\varphi$, i.e.
$\sigma^*(\varphi)=\varphi$. Then the fixed point set $A=\{p\in M
|\;\sigma(p)=p\}$ is an associative 3-fold in $M$.
\end{lem}

For the details of the proof, see \cite{Joy1}. Note that the fixed
point set $A$ is a closed submanifold of $M$. This is because $A$
can be represented as a preimage of 0 under a continuous map,
$\sigma -Id$. This does not contradict our assumption that
$A$ can be open. The reason for this is when we thicken an open
manifold $A$ to obtain the $G_2$-manifold $M$, then $A$ (even if
it is open in $K$) will be always a closed submanifold of $M$.

\vspace{.1in}

Note that there is a similar construction for the coassociative
case \cite[Prop.10.8.5]{Joy1}.

%-------------------------------------------------------------------------
\section{Cayley submanifolds of $\tSpin(7)$-manifolds}

%-------------------------------------------------------------------------
\subsection{\boldmath$\tSpin(7)$-manifolds and the Cayley calibration\boldmath}\label{sec:spin7}
%-------------------------------------------------------------------------

The octonians $\bO = \bR^8$ are equipped with a triple (and quadruple) cross
product.  This cross product defines a 4-form
$\Psi_0(u,v,w,z) = \langle u \times v \times w , z \rangle$.  Given linear
coordinates $x = (x^0 , x^1 , \ldots , x^7 )$ on $\bR^8$,
\begin{eqnarray*}
  \widehat{\Psi_0} & = & \dx^0 \wedge \phi_0 + \ast\phi_0 \\
  & = &
  \dx^{0123} + \dx^{4567} +
  \left( \dx^{01} + \dx^{23} \right) \wedge
  \left( \dx^{45} +\dx^{67} \right) \\ & &
  + \left( \dx^{02} + \dx^{31} \right) \wedge
  \left( \dx^{46} + \dx^{75} \right)
  + \left( \dx^{03} + \dx^{12} \right) \wedge \left( \dx^{74} + \dx^{65} \right) \, .
\end{eqnarray*}

The exceptional $\tSpin(7) \subset \tSO(8)$ is the subgroup preserving
$\Psi_0$ \cite{BrExHol,HL}.

\vspace{.1in}

A $\tSpin(7)$-{\it structure} on $M^8$ is a principle right $\tSpin(7)$-bundle
$\pi : P \to M$.  The elements of $P_x = \pi^{-1}(x)$ are linear isomorphisms
$u : T_x M \to \bR^8$, and the right action is given by
$u \cdot a = a^{-1} \circ u$.  The $\tSpin(7)$-structure induces a
well-defined 4-form $\Psi$ on $M$ via $\Psi_x = u^* \Psi_0$.  As in the $\G2$
case, M also admits a unique metric $g$ and volume form $\ast 1$ for which
$u: T_x M \to R^8$ is an oriented isometry.

\vspace{.1in}

We say $M$ is a $\tSpin(7)$-{\it manifold} when $\Psi$ is closed.
(Equivalently, $\Psi$ is parallel and the $\tSpin(7)$-structure is
torsion-free \cite{BrExHol}.)  In this case, $M$ is Ricci-flat
\cite[10.65]{Be}, and the metric is real analytic in harmonic
coordinates \cite[Th. 5.2]{DK}.  Since $\Psi$ is harmonic it
follows that $\Psi$ is real analytic as well.

\vspace{.1in}

Given a $\tSpin(7)$-manifold, the 4-form
$\Psi$ is a calibration, known as the {\it Cayley calibration}.  The 4-dimensional
submanifolds $i : X^4 \to M$ calibrated by $\Psi$ are the {\it Cayley
submanifolds}.  Cayley submanifolds are plentiful in $\tSpin(7)$-manifolds:
The Cartan-K\"ahler theory implies that given a Cayley plane $E^4 \subset
T_zM$, there exists a Cayley submanifold $X^4$ tangent to $E$.  Moreover,

% --- Lemma: Cayley ------------------------------------------------------
\begin{lem}
\label{lem:cayley}
Every real-analytic 3-dimensional submanifold $Y^3$ of a $\tSpin(7)$-manifold
$(M^8,\Psi)$ lies in a unique Cayley $X^4$.
\end{lem}
%-------------------------------------------------------------------------
\noindent
The flat case $(M,\Psi) = (\bR^8,\Psi_0)$ was proven by Harvey and Lawson
\cite[Th.4.3]{HL}.  The proof follows from Lemma \ref{lem:closed-cayley-ideal} below and the Cartan-K\"ahler Theorem \cite[\S0.4]{BrClbEm}, and is given at the end of
this section.

\vspace{.1in}

The 4-form satisfies the relation \cite[Ch.5, Th.1.28]{HL}
\begin{equation}\label{eqn:Cayley}
  \Psi_0(u,v,w,z)^2 + |\tIm(u \times v \times w \times z)|^2 =
  | u \wedge v \wedge w \wedge z |^2 \, .
\end{equation}
Let $\tau_0$ denote the vector-valued 4-form
$\tau_0(u,v,w,z) = \tIm( u , v , w, z)$.
Given a $\tSpin(7)$-manifold $(M^8,\Psi)$ and a submanifold
$i:X^4 \hookrightarrow M$,
notice that $i^*\Psi = \textrm{d} \textit{vol}_X$ if and only if
$i^*\tau = 0$, where $\tau$ is the vector valued 4-form
$\tau_x = u^* \tau_0$.  Consequently the Cayley submanifolds are the integral
submanifolds of $\tau = 0$.  Let
$\{ \w^0 \, , \, \ldots \, , \, \w^7 \}$ be a local $\tSpin(7)$-coframing.
That is, $\w^j_x = u_x^* \dx^j$, $j=0, \ldots, 7$, for smoothly varying
isometries $u_x : T_x M \to \bR^8$  in $P_x$.  Let $\{ e_j \}$ denote the dual
framing.  Write $\tau = \sum_1^7 \tau^j e_j$.  Then $X$ is Cayley if and only
if $i^*\tau^j = 0$, $j=1, \ldots, 7$.

%--- Lemma ---------------------------------------------------------------
\begin{lem}\label{lem:closed-cayley-ideal}
Let $\cI$ denote the ideal generated algebraically by the $\tau^j$.  Then
$\cI$ is well-defined and closed under exterior differentiation.
\end{lem}
%-------------------------------------------------------------------------

\begin{proof}
That $\cI$ is well-defined is a consequence of the $\tSpin(7)$
invariance of $\tau$.

\vspace{.1in}

To see that $\cI$ is closed, note that (\ref{eqn:Cayley}) and the
fact that $\Psi$ is parallel imply that $\tau$ is also parallel.
Hence
\begin{eqnarray*}
  0 & = &  \nabla\tau \ = \ \nabla \, (\tau^j \ot e_j)
   \ = \  \nabla \tau^j \ot e_j + \tau^j \ot \nabla e_j
   \ = \  \nabla \tau^j \ot e_j - \tau^j \ot \theta^k_j e_k \, \\
  \Longrightarrow \quad
  \nabla \tau^j & = & \tau^k \ot \theta^j_k \, .
\end{eqnarray*}
Above, $\theta$ is the $\mathfrak{spin}(7)$--valued connection
form.  Since the exterior derivative $\td \tau^j$ is the
skew-symmetrization of the covariant derivative $\nabla \tau^j$,
we have $\td \tau^j = \tau^k \wedge \theta^j_k$.
\end{proof}

\medskip

\noindent{\it Proof of Lemma \ref{lem:cayley}.} As in the case of
Lemma \ref{lem:assoc}, the proof is a straightforward application
of the Cartan-K\"ahler Theorem.  See \cite[\S0.4]{BrClbEm} for a review of integral elements, polar spaces and the Cartan-K\"ahler theory in the context.  As $\cI$ is generated by 4-forms,
and $Y$ is of dimension three, $Y$ is trivially an integral
manifold.  Similarly, $V_3(\cI) = \tGr_3(TM)$, and $T_yY$ is
ordinary, for all $y\in Y$.

\vspace{.1in}

In a $\tSpin(7)$ coframing the $\tau^j$ are given by
\begin{eqnarray*}
  \tau^1 & = & (\w^{03} - \w^{12}) \wedge (\w^{46}+\w^{57})
               - (\w^{02} + \w^{13}) \wedge (\w^{47}-\w^{56}) \\
  \tau^2 & = &  (\w^{01} - \w^{23}) \wedge (\w^{47}-\w^{56})
               - (\w^{03} - \w^{12}) \wedge (\w^{45}-\w^{67}) \\
  \tau^3 & = & (\w^{02} + \w^{13}) \wedge (\w^{45}-\w^{67})
               - (\w^{01} - \w^{23}) \wedge (\w^{46}+\w^{57}) \\
  \tau^4 & = & \w^{1234} - \w^{0235} + \w^{0136} - \w^{0127}
               +\w^{0567} -\w^{1467} +\w^{2457} -\w^{3456} \\
  \tau^5 & = & \w^{1235} +\w^{0234} +\w^{0137} +\w^{0126}
               -\w^{1567} -\w^{0467} -\w^{3457} -\w^{2456} \\
  \tau^6 & = & \w^{1236} +\w^{0237} -\w^{0134} -\w^{0125}
               -\w^{2567} -\w^{3467} +\w^{0457} +\w^{1456} \\
  \tau^7 & = & \w^{1237} -\w^{0236} -\w^{0135} +\w^{0124}
               -\w^{3567} +\w^{2467} +\w^{1457} -\w^{0456} \, .
\end{eqnarray*}
Cf. \cite[(6.10)]{McLe}.

\vspace{.1in}

Fix $y$.  Since $\tSpin(7)$ acts transitively on 3-planes
\cite[Th.4]{BrExHol}, we may assume that $T_yY = \tspan\{e_0 , e_1
, e_2 \}$. Then $H(T_y Y) = \tspan\{ e_0 , e_1 , e_2 , e_3 \}$,
and the extension rank is zero.  It follows from the Cartan-K\"ahler theorem that $Y$ lies in a unique Cayley
4-manifold.\hfill\qed

%-------------------------------------------------------------------------
\subsection{\boldmath$\tSpin(7)$-involutions\unboldmath}\label{sec:spin7invol}
%-------------------------------------------------------------------------

There are examples of Cayley submanifolds which are the fixed
point sets of $Spin(7)$ involutions, \cite[Prop.10.8.6 ]{Joy1}.

\vspace{.1in}

As in $G_2$ case, let $\sigma :M\rightarrow M$ be a nontrivial
isometric involution of a $Spin(7)$-manifold $(M,\Psi,g)$
satisfying $\sigma^*(g)=g$ and $\sigma^2=id$, but $\sigma\neq 1$.

\vspace{.1in}

\begin{lem}
\label{involution2} Let $(M,\Psi,g)$ be a $Spin(7)$-manifold and
let $\sigma:M\rightarrow M$ be a nontrivial isometric involution
preserving $\Psi$, i.e. $\sigma^*(\Psi)=\Psi$. Then each connected
component of the fixed point set $A=\{p\in M |\;\sigma(p)=p\}$ is
either a Cayley 4-fold in $M$ or a single point.
\end{lem}

For the details of the proof, see \cite{Joy1}.

\vspace{.1in}

%-------------------------------------------------------------------------
\section{$G$-structures and ideals}\label{sec:Gstr}
%-------------------------------------------------------------------------
The primary purpose of this section is to introduce the principle
objects of interest, and establish notation.  Detailed discussions
may be found in \cite[Ch.II]{KN1} and \cite[\S1]{BrExHol}.
%-------------------------------------------------------------------------
\subsection{\boldmath $G$-structures\unboldmath}
%-------------------------------------------------------------------------
Given a smooth manifold $M$ of dimension $n$, let $\pi : \cF \to
M$ denote its bundle of $\bR^n$--valued coframes.  The fibre
$\pi^{-1}(x) =: \cF_x$ over $x\in M$ is the collection of linear
isomorphisms $u : T_xM \to \bR^n$.  This is a principle right
$\tGL(n)$--bundle with action of $a \in \tGL(n)$ given by $u \cdot
a := a^{-1} \circ u$.  Given a subgroup $G \subset \tGL(n)$, a
{\it $G$-structure} is a principle sub-bundle $\cP \subset \cF$
with structure group $G$.

\vspace{.1in}

\noindent{\it Example.}  When $G = \tSO_n$, there is a unique
Riemannian metric on $M$ for which $\cP$ is the bundle of
(oriented) orthonormal coframes.

\vspace{.1in}

\noindent In the case that $G \subset \tSO(n)$, let $\overline\cP
:= \cP \cdot \tSO(n)$ be the $\tSO(n)$-bundle of orthonormal
coframes.  The corresponding Riemannian metric on $M$ is the {\it
underlying metric of the $G$-structure}. \smallskip
%-------------------------------------------------------------------------
\subsection{Flat structures}
%-------------------------------------------------------------------------
Given a coordinate neighborhood $x : U \to \bR^n$ on $M$, notice
that $\dx$ is a local section $U \to \cF$.  We say the
$G$-structure $\cP$ is {\it flat} when $M$ every $p \in M$ admits
a coordinate chart such that $\dx$ is a local section of $\cP$.

\vspace{.1in}

Clearly, $\cF$ is a flat $\tGL(n)$-structure.  Every orientable
$M$ admits a $\tSL(n)$-structure, given by the volume form.
(Alternatively, every $\tSL(n)$-structure on $M$ uniquely
determines a volume form.)  Because the volume form may always be
expressed locally as $\dx^1 \wedge \cdots \wedge \dx^n$ in some
local coordinate system, every $\tSL(n)$-structure is flat.
%-------------------------------------------------------------------------
\subsection{Connections}
%-------------------------------------------------------------------------
Given a $G$-structure $\pi : \cP \to M$, a tangent vector $v \in
T_u \cP$ is {\it vertical} if $\pi_*(v) = 0$.  A differential
$p$-form $\Omega$ on $\cP$ is {\it semi-basic} if $v \lrcorner\,
\Omega = 0$ for all vertical $v$.  There is a canonically defined,
$\bR^n$--valued semi-basic 1-form $\eta$ on $\cP$:  given $v \in
T_u \cP$,
\begin{displaymath}
  \eta(v) \ := \ u \circ \pi_* (v) \, .
\end{displaymath}

\vspace{.1in}

 The components of $\eta = (\eta^1 , \ldots ,
\eta^n)$ give a basis of the semi-basic 1-forms on $\cP$.

\vspace{.1in}

Let $V_u := \tker \, \pi_* \subset T_u \cP$ denote the vertical
subspace at $u \in \cP$.  A {\it connection} on $\cP$ is a smooth
distribution $H_u \subset T_u \cP$ that is complimentary to $V_u$
and invariant under the right action of $G$.  Equivalently, $\pi_*
: H_u \to T_x M$ is an isomorphism, $x = \pi(u)$; and $(R_a)_* H_u
= H_{u\cdot a}$, where $R_a : \cP \to \cP$ is the map $u \mapsto u
\cdot a$.

\vspace{.1in}

The connection $H$ determines a $\fg$--valued {\it connection
1-form} $\theta$ satisfying $(R_a)^* \theta = \tad(a^{-1}) \theta$
as follows.  Setting $\Theta_{|H_u} \equiv 0$, it remains to
specify $\theta$ on $V_u$.  Every $X \in \fg$ determines a
vertical vector field $X^*$ on $\cP$: given $a(t) \in G$ with
$a(0) = \tId$ and $a'(0) = X$, define $X^*_u = \frac{\td}{\td t} u
\cdot a(t) |_{t=0}$.  These $X^*_u$ span $V_u$, and defining
$\theta(X^*) = X$ determines $\theta$.  Clearly the connection
form is $\fg$--valued.  We leave it to the reader to confirm that
$(R_a)^* \theta = \tad(a^{-1}) \theta$.  Conversely, any
$\fg$--valued 1-form satisfying this condition determines a
connection $H$ via the assignment $H_u := \{ \theta_u = 0 \}$.
%-------------------------------------------------------------------------
\subsection{Torsion}\label{sec:torsion}
%-------------------------------------------------------------------------
It can be shown that a $\fg$--valued 1-form $\theta$ is a
connection form if and only if $\td\eta^j = -\theta^j_k \wedge
\eta^k + T^j_{k\ell} \eta^k \wedge \eta^\ell$, with $T^j_{k\ell} +
T^j_{\ell k} = 0$, \cite[Prop. 8.3.3]{IL}.  The functions
$T^j_{k\ell}$ define a map $T := T^j_{k\ell}
\frac{\partial}{\partial x^j} \ot \dx^k \wedge \dx^\ell : \cP \to
\bR^n \ot \Lambda^2 (\bR^n)^*$, called the {\it torsion of
$\theta$}.  Any other connection 1-form $\tilde \theta$ differs
from $\theta$ by a $\fg$-valued semi-basic 1-form,
$\tilde\theta^j_k = \theta^j_k + c^j_{k\ell} \eta^\ell$.  The
corresponding change in torsion is $\tilde T^j_{k\ell} -
T^j_{k\ell} = c^j_{k\ell} - c^j_{\ell k}$.  In particular, $\tilde
T - T$ takes values in the image of the skew-symmetrizing map
$\delta : \fg \ot (\bR^n)^* \subset \bR^n \ot (\bR^n)^* \ot
(\bR^n)^* \to \bR^n \ot \Lambda^2 (\bR^n)^*$.  This leads to the
definition of the {\it torsion of the $G$-structure $\cP$} as $[T]
: \cP \to h^0(\fg) := (\bR^n \ot \Lambda^2 (\bR^n)^*)/\delta(\fg
\ot (\bR^n)^*)$.

\vspace{.1in}

\noindent{\it Example.} When $\fg = \fso(n)$, it is easy to show
that $h^0(\fso(n)) = \{0\}$.  This is equivalent to the existence
of a torsion-free $g$--compatible connection on a Riemannian
manifold $(M,g)$.  (This is the existence-half of the fundamental
theorem of Riemannian geometry.  The uniqueness-half is equivalent
to $\fso(n)^{(1)} := \tker \, \delta = (\fso(n) \ot (\bR^n)^*)
\cap ( \bR^n \ot S^2 (\bR^n)^*) = \{0\}$.  In general, $\fg^{(1)}$
records the changes in connection that preserve torsion.)  When
$\fg \subset \fso(n)$, we have $h^0(\fg) = ( \fso(n) / \fg ) \ot
(\bR^n)^*$.

\vspace{.1in}

%-------------------------------------------------------------------------
\subsection{Torsion-free
\boldmath $G$-structures\unboldmath}
%-------------------------------------------------------------------------
We say that $\cP$ is {\it torsion-free} when $[T] \equiv 0$, and
$M$ is a {\it $G$-manifold}.  In the case that $G \subset
\tSO(n)$,   let $i : \cP \to \overline \cP$ denote the inclusion,
and $\overline H$ the Levi-Civita connection on $\overline\cP$.
It is not difficult to see that $\cP$ is torsion-free if and only
if $\overline H \subset i_* T \cP$.  Equivalently, $\cP$ is
preserved under parallel transport in $\overline\cP$.

\vspace{.1in}

Torsion may be viewed as a first-order obstruction to flatness.
Here is one way to see this.  Suppose $M$ carries a $G$-structure
$\cP$.  Let $x : U \to \bR^n$ be a coordinate system about $z \in
M$.  We may assume that the local section $\dx : U \to \cF$
satisfies $\dx_{z} \in \cP_{z}$. The coordinates define a local,
flat $G$-structure $\cP_0 := \dx\cdot G$ over $U$.  The following
lemma is well-known.

%-------------------------------------------------------------------------
\begin{lem}\label{lem:torsion-free}
The $G$-structure $\cP$ is torsion-free if and only if for all $z
\in M$, there exist local coordinates $(x,U)$ such that $\cP$ and
$\cP_0$ are tangent at $\dx_z$.
\end{lem}
%-------------------------------------------------------------------------

%-------------------------------------------------------------------------
\subsection{\boldmath $G$-structures \unboldmath
as sections and 1-flatness}\label{sec:1-flat}
%-------------------------------------------------------------------------
Let $S = \cF/G$, and consider the bundle $\overline\pi : S \to M$,
\begin{center}
\setlength{\unitlength}{1cm}
\begin{picture}(2,3)
  \put(0.3,2.7){$\cF$}
  \put(0.4,2.5){\vector(0,-1){2}}
  \put(0.2,0){$M$}
  \put(0.8,2.8){\vector(1,-1){1}}
  \put(1.9,1.45){$S$}
  \put(1.8,1.3){\vector(-1,-1){1}}
  \put(0,1.5){$\pi$}
  \put(1.4,0.6){$\overline \pi$}
  \put(1.4,2.4){$\rho$}
\end{picture}
\end{center}
Notice that $G$-structures on $M$ are in one-to-one correspondence
with $\overline\pi$-sections $\sigma : M \to S$.

\vspace{.1in}

We say $\sigma : M \to S$ is {\it flat} when the corresponding
$G$-structure is flat.  The section $\sigma$ is {\it 1-flat} if it
is flat to first-order at every point.  That is, if every $x\in M$
admits an open neighborhood $U$ carrying a flat $G$-structure with
corresponding section $\sigma_0 : U \to S$ such that $\sigma(x) =
\sigma_0(x)$, and $\sigma(M)$ and $\sigma_0(U)$ are tangent at
$\sigma(x)$.  By Lemma \ref{lem:torsion-free}, {\it $\cP$ is
torsion-free if and only if $\sigma$ is 1-flat}.

%-------------------------------------------------------------------------
\subsection{Admissible groups}
%-------------------------------------------------------------------------
Given $G \subset \tSO(n)$, let $\Lambda^*(\bR^n)^G$ denote the
$G$-invariant constant coefficient differential forms on $\bR^n$.
We say $G$ is {\it admissible} if it is the subgroup of $\tGL(n)$
fixing all the forms in $\Lambda^*(\bR^n)^G$.  (A priori, this subgroup contains $G$.)  In \cite{BrExHol}
Bryant showed that $G_2$ and $\tSpin(7)$ are admissible.  The ring
$\Lambda^*(\bR^n)^{\G2}$ is generated by $\varphi_0$ and
$\ast\varphi_0$ (cf. \S\ref{sec:G2}), and
$\Lambda^*(\bR^n)^{\tSpin(7)}$ by $\Psi_0$ (cf.
\S\ref{sec:spin7}).
%-------------------------------------------------------------------------
\subsection{A differential ideal on
\boldmath $S$ \unboldmath}\label{sec:ideal}
%-------------------------------------------------------------------------
This subsection and the following borrow heavily from
\cite{BrClbEm}.

\vspace{.1in}

 Every $p$-form $\alpha$ on $\bR^n$ defines a
semi-basic $p$-form $\hat\a$ on $\cF$ via
\begin{displaymath}
\hat \alpha_u ( v_1 , \ldots , v_p ) \ := \
  \alpha ( \eta(v_1) , \ldots , \eta(v_p) ) \, .
\end{displaymath}

\vspace{.1in}

When $\a$ is $G$-invariant, $\hat\a$ is invariant under the
right-action of $G$ on $\cF$ and therefore descends to a
well-defined $p$-form on $S$, also denoted by $\hat\a$.  Given a
$G$-structure $\cP$ with corresponding section $\sigma : M \to S$,
the pull-back $\sigma^* \hat\a$ defines a $p$-form $\a_\sigma$ on
$M$.  Recall that $\sigma$ is torsion-free if and only if the
$G$-structure is preserved under parallel transport by the
underlying Levi-Civita connection.  In particular, $\a_\sigma$ is
parallel, and therefore closed, if $\sigma$ is torsion-free.
Whence $\sigma^*(\td \hat\a) = 0$.

\vspace{.1in}

We denote by $\cI$ both the ideal on $\cF$ and the ideal on $S$
that is generated algebraically by $\td \hat\a$, $\a \in
\Lambda^*\bR^n$.  Graphs of torsion-free $\sigma : M \to S$ are
necessarily integral manifolds of $\cI$. The converse need not
hold; see \cite[\S0.5.5]{BrClbEm} for an example.
%-------------------------------------------------------------------------
\subsection{Strong admissibility}\label{sec:str_adm}
%-------------------------------------------------------------------------
Given $k \le n$, let $V(\cI,\overline\pi) \subset \tGr_k(TS)$
denote the $k$-dimensional integral elements $E \subset T_sS$ that
are $\overline\pi$-transverse; that is, the projection
$\overline\pi_* : E \to T_{\overline\pi(s)} M$ is injective.  As
noted above, $V_n(\cI,\overline\pi)$ contains the set of
$n$-planes tangent to the graph of a torsion-free section
$\sigma$.  When $G$ is admissible, and $V_n(\cI,\overline\pi)$
consists of exactly these tangent planes, then we say $G$ is {\it
strongly admissible}. As a result any section $\sigma:M\to S$ whose
image in $S$ is an integral manifold of $\cI$ is necessarily
torsion-free.  Both $G_2$ and $\tSpin(7)$ are strongly admissible
\cite{BrExHol}.

\vspace{.1in}

Recall from \S\ref{sec:torsion} that the torsion of
a $G$-structure $\sigma : M \to S$ lives in $\fh^0(\fg)$.  Since
$V_n(\cI,\overline\pi)$ contains the tangent planes to torsion-free
$\sigma(M)$, and torsion is a first-order invariant, we must have
\begin{displaymath}
  \tcodim( V_n(\cI,\pi) , \tGr_n(T\cF) ) \le \tdim\, \fh^0(\fg) \, ,
\end{displaymath}
with equality precisely when $G$ is strongly admissible.

%-----------------------------------------------------------------------
\subsection{Integral elements on \boldmath $S$ and $\cF$\boldmath}
\label{sec:int_elem}
%-------------------------------------------------------------------------
Define $V_k(\cI,\pi) \subset \tGr_k(T_u \cF)$ to be the
$k$-dimensional integral elements of $\cI$ that are
$\pi$-transverse.  Observe that $\rho_* : V_k(\cI,\pi) \to
V_k(\cI,\overline\pi)$ is a surjection. Given $E \, , \ E' \in
V_k(\cI,\pi)$, we have $\rho_*(E) = \rho_*(E')$ if and only if $E
\equiv E'$ mod $\fg_u$. Here $\fg_u := \tker ( \rho_{*|T_u\cF} )$.
(Alternatively, $\fg_u \subset V_u$ is the vertical subspace of
$T_u \cF$ identified with $\fg$ under the right action of $G$ at
$u\in\cF$.)  In particular, for fixed $E$ the set of all such $E'$
is naturally identified with $\textrm{Hom}(E , \fg_u)$.  This set
is of dimension $k\, \tdim(G)$.  As $\tdim \, T\cF - \tdim \, TS =
n\, \tdim(G)$, we have
\begin{equation}\label{eqn:codim}
  \tcodim( V_n(\cI,\overline\pi) \, , \, \tGr_n(TS) ) \ = \
  \tcodim( V_n(\cI,\pi) \, , \, \tGr_n(T\cF) ) \,
\end{equation}
in the case that $k=n$.

\vspace{.1in}

It is straightforward to check that, given $E \in V_k(\cI,\pi)$,
the polar spaces satisfy $H(E) \ = \
(\rho_*)^{-1} H( \rho_* E )$.  In particular, given an integral
flag $F = \{E_j\}_{j=0}^n$ in $T_u\cF$ and the  and the
corresponding flag $\overline F = \{ \overline E_j = \rho_* E_j
\}_{j=0}^n$ in $T_{\rho(u)}S$, we have
\begin{equation}\label{eqn:c}
  c_j(F) \ = \ c_j(\overline F) \qquad \textrm{and} \qquad
  c(F) \ = \ c(\overline F) \, .
\end{equation}

Finally, it is not difficult to describe the set $V_n(\cI,\pi)
\subset \tGr_n(T\cF)$. Given a $\pi$-transverse $E \in
\tGr_n(T\cF)$ the canonical 1-forms $\eta^j$ span $E^*$.  In
particular, when restricted to $E$ the connection 1-forms
$\theta^j_k$ may be expressed as linear combinations of the
$\eta^j$, $\theta^j_k = p^j_{k\ell} \eta^\ell$. The $p^j_{k\ell}$
are functions on $\cF$, and $p^j_{k\ell}(u)$ parameterizes the
open set of $\pi$-transverse $n$-planes $E \in \tGr_n(T_u\cF)$.
Now, given $\a = a_I \dx^I \in \Lambda^p(\bR^n)^*$, $I = \{ i_1 \,
, \ \cdots \, , \ i_p \}$ a multi-index, we have  $\hat\a = a_I
\eta^I$.  When restricted to $E \in \tGr_n(T\cF,\pi)$, $\td\eta^j
= -\theta^j_k \wedge \eta^k = p^j_{k\ell} \eta^k \wedge
\eta^\ell$.  We see that the equation $\td\hat\alpha = 0$ is a set
of linear conditions on the $p^j_{k\ell}$, and that $V_n(\cI,\pi)$
is a submanifold of $\tGr_n(T\cF)$. (The exterior differential
system $\cI$ with independence condition $\eta^1 \wedge \cdots
\wedge \eta^n \not= 0$ is {\it in linear form}.) In particular,
each $E \in V_n(\cI,\pi)$ is ordinary.

%-------------------------------------------------------------------------
\subsection{Canonical flags and regular presentations}\label{sec:flags}
%-------------------------------------------------------------------------

To each $n$-dimensional integral element $E_n \in V_n(\cI,\pi)$ at
$u \in\cF$ we may canonically associate a flag $F = \{ E_0 \subset
E_1 \subset \cdots \subset E_n \}$ by
\begin{displaymath}
  E_k \ := \ \{ v \in E_n \ | \ \eta^j(v) = 0 \ \forall \, j > k \} \, .
\end{displaymath}

\vspace{.1in}

The polar spaces are $H(E_k) = E_n + (\fh_k)_u$,
where the $\fh_k$ are defined as follows.  Let $i_k: \bR^k
\hookrightarrow \bR^n$ denote the natural inclusion, and set
\begin{displaymath}
  \fh_k \ := \ \{ x \in \fgl(n) \ | \ i_k{}^*(x . \alpha) = 0 \, , \ \forall
    \alpha \in \Lambda^*(\bR^n)^G \} \, .
\end{displaymath}

\vspace{.1in}

Note that $\fh_k$ contains $\fg$, $\fh_{k+1}$ and the space
$M_{n,k}\bR$ of $n$-by-$n$ matrices whose first $k$ columns are
zero. When $G$ is admissible, $\fh_n = \fg$.

 \vspace{.1in}

Cartan's Test implies that
\begin{displaymath}
  \sum_{j=0}^{n-1} c_j \ = \
  \sum_{j=0}^{n-1} \tcodim(\fh_j , \fgl(n)) \ \le \
  \tcodim( V_n(\cI,\pi) , \tGr_n(T\cF) ) \, .
\end{displaymath}

\vspace{.1in}

A strongly admissible group $G$ is {\it regularly presented} when
equality holds.  We will see that both $G_2$ and $\tSpin(7)$ are
regularly presented. When $G$ is regularly presented every $E \in
V_n(\cI,\pi)$ is the terminus of a regular flag $F$.  It follows
from (\ref{eqn:codim}, \ref{eqn:c}) and Cartan's Test that every
$\overline E_n \in V_n(\cI,\overline\pi)$ is also the terminus of
a regular flag, $\{ \overline E_j := \rho_*(E_j)\}$.

\vspace{.1in}

%We have the following existence result of Bryant \cite{BrClbEm}.
%\begin{cor}[Bryant]
%If $G\subset \tSO(n)$ is conjugate to a regularly presented
%subgroup of $\tSO(n)$, then every $\overline E \in
%V_n(\cI,\overline\pi)$ is tangent to the graph of some local
%section $\sigma: U \to S$ defined on an open $U \subset M$ and
%satisfying $\sigma^*(\cI) = 0$.  Moreover, when $G$ is strongly
%admissible, the section is torsion-free.
%\end{cor}

%-------------------------------------------------------------------------
\section{Real algebraic and real analytic structures}
%-------------------------------------------------------------------------

In this section we briefly review the basics of real algebraic
sets. For more information see \cite{AK1,AK2,AK3}.

\vspace{.1in}

A real algebraic set is the set of solutions of polynomial
equations in real variables, a set $V$ of the form
$V(I)=\{x\in \mathbb{R}^n \ | \ p(x)=0$, for all $p\in I\}$ where $I$ is
a set of polynomial functions $p : \mathbb{R}^n\rightarrow
\mathbb{R}$.

\vspace{.1in}

A point $x$ in an algebraic set $V\subset \mathbb{R}^n$ is called
{\it nonsingular of codimension $k$ in $V$} if there are polynomials
$p_i$, $i\in\{1,\ldots,k\}$, and a neighborhood $U\subset\bR^n$
of $x$ so that $p_i(V)=0$ and

\begin{itemize}
\setlength{\parsep}{0pt} \setlength{\itemsep}{0pt} \item[{\rm(i)}]
$\ V\cap U=U\cap\bigcap_{i=1}^{k}p_i^{-1}(0)$ \item[{\rm(ii)}] the
gradients $\nabla p_i$, are linearly independent on $U$.
\end{itemize}

\vspace{.1in}

Define $\tdim V$ to be the maximum of $n-k$ over all nonsingular
$x \in V$. Then for an algebraic set $V$,
$$\mathit{Nonsing}(V):=\{x\in V \ | \ x \textrm{ is nonsingular of
dimension } \tdim V \}$$
 and $\mathit{Sing}(V) := V \backslash
\mathit{Nonsing}(V)$. We say that an algebraic set $V$ is {\it
nonsingular} if $\mathit{Sing}(V)=\emptyset$.

\vspace{.1in}

Nash \cite{N} proved that every smooth, closed manifold is a
topological component of a nonsingular algebraic set, and
conjectured that every smooth, closed manifold is a nonsingular
algebraic set. Tognoli verified the Nash's conjecture.

\begin{thm}[Nash-Tognoli \cite{N,T,AK3}]\label{thm:tognoli}
Let $M$ be a smooth, closed manifold.  Then there exists a
nonsigular algebraic set $V$ and a diffeomorphism $\phi : M \to
V$.
\end{thm}

In \cite{AK1, AK2, AK3}, Akbulut and King generalized Nash's
theorem to interiors of compact manifolds and proved that the
interior of a smooth, compact manifold $M$ is diffeomorphic to a
nonsingular real algebraic set $V$ which is properly imbedded in
$\mathbb{R}^n$ for some $n$. This established a one-to-one
correspondence between interiors of compact smooth manifolds and
nonsingular real algebraic sets.

\vspace{.1in}

Note that, one can use Akbulut-King's result and the Real Analytic
Implicit Function Theorem, Theorem \ref{thm:implicit}, to find
real analytic metrics on interiors of compact manifolds. In this
paper, we won't use this fact as we first double our manifold to
obtain a closed manifold, construct the real analytic metric and
then take its restriction.

\vspace{.1in}

In \cite{Bochner}, Bochner proved that on a closed, real analytic
manifold the real analytic differential forms are dense in the
smooth forms in the uniform topology.  Next, we show that one can
obtain Bochner's result using real algebraic theory developed by
Akbulut and King.

\begin{thm} \label{thm:altboch} Every smooth, closed manifold $X$ can be made a nonsingular real
algebraic variety $V$.  Every smooth differential form on $V$ can be
approximated by a real analytic differential form.
\end{thm}

\begin{proof}

Let $X$ be a smooth, closed manifold. By Nash-Tognoli, it can be
made a nonsingular real algebraic variety $V$. Now, we show that
every smooth differential form on $V$ can be approximated by a
real analytic differential form. In \cite{AK3}, it was shown that
for a nonsingular algebraic set $V$ with dimension $k$, the
classifying Gauss map $\rho:V\rightarrow G(k,n)$ of the tangent
bundle $TV\rightarrow V$ is an entire rational map. Now, let
$E(k,n)$ be the universal bundle over the Grassmannian variety of
$k$ planes in $\mathbb{R}^n$:

\begin{equation*}
\begin{split}
&E(k,n)=\{(A,v)\in \mathcal{M}_{\mathbb R}(n)\times \mathbb{R}^n
|\;A\in G(k,n), \; Av=v \} \\
&\;\;\;\downarrow \pi \\
&G(k,n)=\{A\in \mathcal{M}_{\mathbb R}(n)|\; \; A^t=A, \;\; A^2=A,
\;\; trace(A)=k \}
\end{split}
\end{equation*}

\noindent where $\mathcal{M}_{\mathbb R}(n)$ denotes $n\times n$
real matrices. The tangent bundle $TV$ can be identified with the
pullback bundle $\rho^*E =\{ (x,v)\in V\times E | \;\;
\rho(x)=\pi(v)\}$ where $\rho$ and $\pi$ are both algebraic.

\vspace{.1in}

Then we have the following diagram:

\begin{equation*}
\begin{split}
&TV\;\;\;\;\;\;\;\;\;\;\;\; E(k,n) \\
& \;\; \downarrow \;\;\;\;\;\;\;\;\;\;\;\;\;\;\downarrow \pi \\
&\;\;\;V\;\;\longrightarrow \;\; G(k,n)
\end{split}
\end{equation*}

\noindent where $V$, $TV$ are nonsingular algebraic sets. Sections
$s:V\rightarrow TV\cong\rho^*E$ are given by $s(x)=(x, f(x))$, for
some function $f:V\rightarrow E$. This means that for real
analytic approximation of the sections $s$, it is sufficient to
find real analytic approximations of $f$.

\vspace{.1in}

Now, $V\subset \mathbb{R}^\ell$ and $E(k,n)\subset \mathbb{R}^m$ are
both nonsingular algebraic sets for some $\ell,m$. By the Weierstrass
Approximation Theorem, a real valued smooth map from an open
subset of  $\mathbb{R}^\ell$ can always be approximated by
polynomials. Denote this polynomial approximating $f$ as $F$. Even
though $F$ may not map $V$ to $E$, we can always take the
projection $\pi$ from the image of $F$ to $E$. This map, which is
from a tubular neighborhood of an algebraic variety $E$ to $E$,
is real analytic. So the composition of $F$ and
$\pi:tubular(E)\subset\mathbb{R}^N\rightarrow E$ yields a real analytic approximation of $f$, and thus the section $s=(x, f(x))$.

\vspace{.1in}

Since the cotangent and tangent bundles are the dual bundles, the
same real analytic approximation holds for differential forms.
\end{proof}

\vspace{.1in}

An important property of real analytic functions is that the
inverse of a real analytic function is also real analytic.

\begin{thm}[Real analytic inverse function theorem]
Let $F$ be real analytic in a neighborhood of $a=(a_1,a_2,...,
a_n)$ and suppose that its derivative at $a$, $DF(a)$, is
nonsingular. Then $F^{-1}$ is defined and real analytic in a
neighborhood of $F(a)$.

\end{thm}

The proof of this theorem follows from a special case of the
Cauchy-Kowalewsky Theorem \cite{Kr1,Kr2}.

\vspace{.1in}

As an important corollary, we obtain the implicit function theorem
in the analytic setting.

\begin{thm}[Real analytic implicit function theorem]
\label{thm:implicit} Suppose $F:\mathbb{R}^{n+m}\rightarrow
\mathbb{R}^m$ is real analytic in a neighborhood of $(x_0,y_0)$,
for some $x_0\in \mathbb{R}^n$ and some $y_0\in \mathbb{R}^m$. If
$F(x_0,y_0)=0$ and the $m \times m$ matrix with entries
$\frac{\partial F_i}{\partial y_j}${\scriptsize{$(x_0,y_0)$}}
is nonsingular, then there exists a function $f:\mathbb{R}^n
\rightarrow \mathbb{R}^m$ which is real analytic in a neighborhood
of $x_0$ and is such that $F(x,f(x))=0$
holds in a neighborhood of $x_0$.

\end{thm}

Assume that $V \subset \bR^n$ is a algebraic set, and define $Z =
\mathit{Nonsing}(V)$.  Let $g_0$ denote the canonical Euclidean
metric on $\bR^n$.  The Implicit Function Theorem for real
analytic maps \cite{Kr1,Kr2} implies that the restriction of $g_0$
to $Z$ is a real analytic Riemannian metric.

\vspace{.1in}

Let $K$ be a smooth, compact manifold.  Then the double
$\mathit{doub}(K)$ is  smooth and closed.  From Theorem
\ref{thm:tognoli} and Theorem \ref{thm:implicit}, we deduce the
following lemma.

\begin{lem}\label{lem:analyticstr}

Let $\mathit{doub}(K)$ be the double of a smooth, compact manifold
$K$.  Then $\mathit{doub}(K)$ admits a compatible real analytic
structure and a real analytic metric $g$.
\end{lem}

%-------------------------------------------------------------------------
\section{The associative embedding: proof of Theorems \ref{thm:assoc} \& \ref{thm:openassoc}}
\label{sec:assocproof}
%-------------------------------------------------------------------------
Bryant has shown that the group $G_2$ is admissible
\cite[Prop.1]{BrExHol}.  Now consider the differential system
$\cI$ of \S\ref{sec:ideal}.  Here $n=7$ and the indices $j,k$
range over $1 \, , \, \ldots \, , \, 7$.  On $\cF$, the ideal is
generated by the 4-form $\td \widehat \varphi_0$ and the 5-form
$\td(\widehat{\ast\varphi_0})$ where
\begin{eqnarray*}
  \widehat{\varphi_0} & = & \eta^{123}
  + \eta^1 \wedge \left( \eta^{45} + \eta^{67} \right)
  + \eta^2 \wedge \left( \eta^{46} - \eta^{57} \right)
  + \eta^3 \wedge \left( - \eta^{47} - \eta^{56} \right) \\
  \widehat{\ast\varphi_0} & = & \eta^{4567}
  + \eta^{23} \wedge \left( \eta^{45} + \eta^{67} \right)
  + \eta^{31} \wedge \left( \eta^{46} - \eta^{57} \right)
  + \eta^{12} \wedge \left(  - \eta^{47} - \eta^{56} \right) \, .
\end{eqnarray*}

\vspace{.1in}

Given a $\fgl(7)$-valued connection form $\theta^j_k$ on $\cF$, we
have $\td\eta^j = -\theta^j_k\wedge\eta^k$.  Let $p^j_{k\ell}(u)$
be the functions parameterizing $\tGr_7(T_u\cF,\pi)$ introduced in
\S\ref{sec:int_elem}: on any $E \in \tGr_7(T_u\cF,\pi)$ we have
$\theta^j_k = p^j_{k\ell} \eta^\ell$. Then the equations $\td
\widehat\varphi_0 = 0 = \td\widehat{\ast\varphi}_0$ defining
$V_7(\cI,\pi)$ in $\tGr_7(T\cF)$ are linear conditions on the
parameters $p^j_{k\ell}$.  The first equation is equivalent to 35
(independent) linear constraints on the $p^j_{k\ell}$, and the
second equation imposes an additional 14.  Hence
\begin{displaymath}
  \tcodim( V_7(\cI,\pi) , \tGr_7(\cF) ) = 49 \, .
\end{displaymath}
Moreover, by work of Fernandez and Gray \cite{FG} we know that
$\tdim \, h^0(\fg_2) = 49$.  Whence $\G2$ is strongly admissible
(\S\ref{sec:str_adm}).

\vspace{.1in}

Fix $E \in V_7(\cI,\pi)$, and let $F$ denote the canonical flag of
\S\ref{sec:flags}.  Next we compute the sequence of polar spaces.
Since $(\Lambda^*\bR^7)^{G_2}$ contains no 1- or 2-forms, we have
$\fh_0 = \fh_1 = \fh_2 = M_7 \bR \simeq \bR^{49}$.  Next,
$i_3^*(x.\varphi_0) = (x^1_1+x^2_2+x^3_3) \, \dx^1 \wedge \dx^2
\wedge \dx^3$, so that $\fh_3 = \{ x \in M_7 \bR \ | \ x^1_1 +
x^2_2 + x^3_3 = 0 \} $. Similarly, $i_4^*(x.\varphi_0) = 0 =
i_4^*(x.\ast\varphi_0)$ implies that $\fh_4 \subset \fh_3$ is
given by the four additional equations
\begin{displaymath}
  0 \ = \ x^1_4 - x^6_3 - x^7_2 \ = \
  x^2_4 + x^5_3 + x^7_1 \ = \
  x^3_4 - x^5_2 + x^6_1 \ = \
  x^5_1 + x^6_2 - x^7_3 \, .
\end{displaymath}

Continuing in this fashion we find that $\tcodim(\fh_5) = 15$, $\tcodim(\fh_6) = 13$, and $\fh_7 = \fg_2$ so that $\tcodim(\fh_7) = 35$.

Whence the polar space codimensions are $(c_0, c_1, \ldots , c_6)
= (0, 0, 0, 1, 5, 15, 28)$ and $\sum c_j = 49$.  In particular, $\G2
\subset \tSO(7)$ is regularly presented (\S\ref{sec:flags}).
Additionally, Cartan's Test implies that
$V_7(\cI,\pi)$ is a codimension 49 submanifold of $\tGr(\cF,\pi)$,
and each $E \in V_7(\cI,\pi)$ is the terminus of a (canonical)
regular flag $F$.

\vspace{.1in}

This completes the necessary preliminaries for $G_2$.  For Theorem
\ref{thm:assoc} we assume that $(K^3,g)$ is a closed, oriented
real analytic Riemannian 3-manifold. As an oriented 3-manifold $K$
is smoothly parallelizable by a result of Wu \cite{MilSta}. Using
Bochner's result, \cite{Bochner}, or Theorem \ref{thm:altboch}, we
can conclude that $K$ admits a real analytic parallelization.

\vspace{.1in}

In the case of Theorem \ref{thm:openassoc}, invoke Lemma
\ref{lem:analyticstr} to endow $\mathit{doub}(K)$ with a
compatible real analytic Riemannian structure
$(\mathit{doub}(K),g)$.  (If $K$ is closed, then $\mathit{doub}(K)
= K$.)  As above, $\mathit{doub}(K)$ admits a real analytic
parallelization.

\vspace{.1in}

The rest of the argument applies to both theorems.  Let $A =
\mathit{int}(K)$.  (In Theorem \ref{thm:assoc}, $A = K$.)  Assume $A$ is connected, else apply the theorem to each connected component individually.
Restrict the Riemannian metric $g$ and real analytic
parallelization to $A$.  The Gramm-Schmidt process yields an
orthonormal parallelization, and 1-forms $\w_1$, $\w_2$ and $\w_3$
such that
\begin{displaymath}
  g \ = \ \w_1{}^2 + \w_2{}^2 + \w_3{}^2 \, ,
\end{displaymath}
and $\tvol_g = \w_1 \wedge \w_2 \wedge \w_3$.

\vspace{.1in}

Let $M = A \times \bR^4$, and let $y=(y^4,y^5,y^6,y^7)$ be linear
coordinates on $\bR^4$.  Regard the $y^j$ as functions on $M$ and
identify $A$ with the 0-section $A \times \{0\}$.  The 1-forms $\{
\w, \dy \}$ form a coframing of $M$ and define a global section $s
: M \to \cF$.  The corresponding trivialization of $\pi : \cF \to
M$ is given by associating to each $ u \in \cF_z$ the unique $g =
g(u) \in \tGL(7)$ such that $u = g^{-1} \circ s(z)$.  With respect
to the trivialization, the canonical 1-forms $\eta = (\eta^j)$ are
given by
\begin{displaymath}
  \eta_{(z,g)}(v) \ = \ g^{-1}
  \left( \begin{array}{c}
      \w(\pi_*v) \\
      \dy(\pi_*v)
  \end{array}\right) \ , \quad  v \in T_{z,g} (M \times \tGL(7)) \, .
\end{displaymath}
It will be convenient notationally to identify $\cF$ with the
trivialization $M \times \tGL(7)$.

\vspace{.1in}

Define an involution $r:M\to M$ by $r(p,y) = (p,-y)$, $(p,y) \in A \times
\bR^4 = M$.  Lift $r$ to an involution, also denoted by $r$, of
$\cF$ by defining $r(u) = r^*(u)$.  That is, given $u : T_zM \to
\bR^7$ in $\cF_z$ we define $r(u) : T_{r(z)} M \to \bR^7$ in
$\cF_{r(z)}$ to be the map sending $v \in T_{r(z)}M$ to $u(r_*
v)$.  With respect to the trivialization $\cF \simeq A \times
\bR^4 \times \tGL(7)$, we have $r(p,y,g) = (p,-y,Rg)$, where
\begin{displaymath}
  R \ = \
  \left( \begin{array}{cc}
    I_3 & 0 \\
    0 & -I_4
  \end{array} \right) \, .
\end{displaymath}
Notice that $R \in \G2$, so that $r:\cF\to\cF$ preserves the
$\rho$-fibres, and $r$ descends to a well-defined involution of
$S$.  Also, $\pi\circ r = r \circ\pi$ implies that $r^*\eta =
\eta$, so that
\begin{displaymath}
  r^*\left( \widehat{\varphi_0} \right) \ = \ \widehat{\varphi_0}
  \quad \textrm{and} \quad
  r^*\left( \widehat{\ast\varphi_0} \right) \ = \ \widehat{\ast\varphi_0} \, .
\end{displaymath}
Whence $r^*\cI = \cI$, and $r$ carries integral manifolds of $\cI$
to integral manifolds of $\cI$.

\vspace{.1in}

We are now ready to apply the Cartan-K\"ahler theorem to prove
Theorem \ref{thm:assoc}.  Define a lift $f_3: A \to S$ by
$f_3(p,0) = \rho \, (p,0,I_7)$, and let $X_3$ denote the image.
Since $X_3$ is three-dimensional, and $\cI$ is generated by a 4-
and 5-form, $X_3$ is trivially an integral manifold.  Because $R
\in \G2$, $X_3$ lies in the fixed locus of $r$.  We will use the
Cartan-K\"ahler theorem to thicken (in four
steps) $X_3$ to a seven-dimensional $r$-invariant integral
manifold that projects diffeomorphically onto a neighborhood $N$
of $A \subset M$.  Moreover, the induced $\G2$-structure on $N$
will have the properties that (i) $A \subset N$ is associative,
and (ii) the metric induced on $N$ by the $G_2$-structure agrees
with $g$ when restricted to $A$.  The construction is repetitive and very similar to that of \cite{BrClbEm}, so after detailing Steps 1 and 2 below, we will sketch the remaining steps.

\vspace{.1in}

\smallskip

{\it Step 1: Thicken $X_3$ to a 4-dimensional integral manifold
$X_4$ of $\cI$}.  Let $z = f_3(p)$.  To compute the polar space
$H(T_zX_3)$, note that $T_z X_3 = \rho_* T_p$, where $T_p$ is the
3-plane tangent to the lift $\{ (p,0,I_7) \ | \ p \in A \}$ in
$\cF$. There exists an $E_7 \in V_7(\cI,\pi)$ containing $T_p$.
Given this, it is clear that $T_p = E_3$ in the canonical regular
flag $F$ terminating in $E_7$.  Thus $T_z X_3$ is $\overline E_3$
in the canonical regular flag $\overline F$ (cf.
\S\ref{sec:flags}).

\vspace{.1in}

To see that such an $E_7$ exists, recall that $(p^j_{k\ell}) \in
\bR^{7^3}$ parameterizes the open set of $\pi$-transverse $E \in
\tGr_7(T_u\cF, \pi)$, and $V_7(\cI,\pi)$ is a linear subspace of
$\bR^{7^3}$ of codimension 49. The condition that $T^3_p$ lie in
some $E_7$ holds if we are free to specify the values of
$p^j_{k1}, p^j_{k2}, p^j_{k3}$ in $V_7(\cI,\pi)$. It can be checked
that these variables are independent in $V_7(\cI,\pi) \subset
\bR^{7^3}$, so such a specification is always possible.

\vspace{.1in}

We may now compute $\tdim \, H(T_{f_3(p)}X_3) = \tdim\,
H(\overline E_3) = 41$. Hence $X_3$ is a regular integral manifold
of extension rank 37.  Note that
\begin{displaymath}
  \tdim \, S \ = \ 42 \, ;
\end{displaymath}
so to apply the Cartan-K\"ahler Theorem we need to construct a
5-dimensional manifold $Z_3$ that contains $X_3$ with tangent
space at $z \in X_3$ transverse to $H(T_zX_3)$.

\vspace{.1in}

Let $W_1 \subset M_7\bR$ be the 1-dimensional subspace of matrices
of the form
\begin{displaymath}
  \left( \begin{array}{cc}
     x_1 I_3 & 0 \\
     0 & 0
  \end{array} \right) \, , \quad x_1 \in \bR \,  .
\end{displaymath}
Notice that $W_1 \cap \fh_3 = \{0\}$ and $R W_1 = W_1$.  Since
$\fg_2 \subset \fh_3$, the affine space $I_7+W_1$ intersects $\G2$
transversely at $I_7 \in \tGL(7)$.  Hence, there is a neighborhood
$U_1$ of $0$ in $W_1$ such that the map $U_1 \to \tGL(7)/G_2$
sending $x \mapsto (I+x)G_2$ is an embedding.

\vspace{.1in}

Define a 5-dimensional manifold $Z_3 \subset S$ by
\begin{displaymath}
  Z_3 \ := \ \left\{ \rho \left( p , (y^4, 0, 0, 0) , I_7 + x \right) \ |
  \ p \in A \, , \ y^4 \in \bR \, , \ x \in U_1 \right\} \, .
\end{displaymath}
As constructed $Z_3$ contains $X_3$, is $r$-invariant and
$H(T_zX_3) \cap T_zZ_3$, $z\in X_3$, is of dimension 4.  The
Cartan-K\"ahler theorem concludes that there exists a real
analytic, 4-dimensional integral manifold $Y_4 \subset Z_3$
containing $X_3$. Since $r^*\cI = \cI$, $r(Y_4)$ is also an
integral manifold.  And since $X_3$ and $Z_3$ are $r$-invariant,
we have $X_3 \subset r(Y_4) \subset Z_3$.  By the uniqueness part
of the Cartan-K\"ahler theorem the $r$-invariant $X_4 := Y_4 \cap
r(Y_4)$ is also a 4-dimensional integral manifold of $\cI$.

\vspace{.1in}

Given $z \in X_3$ note that the 4-plane $T_zX_4 = H(T_zX_3) \cap
T_zZ_3$ is (i) $\overline\pi$-transverse, and (ii) the $\overline
E_4$ of a regular flag $\overline F$.  Transversality implies that
a neighborhood of $X_3$ in $X_4$ projects diffeomorphically onto a
neighborhood $N_4$ of $A$ in $A \times \bR \subset M$.  Shrinking
$X_4$ if necessary, we may assume that $X_4$ is image of a section
$N_4 \to S$.  Item (ii) implies that $T_zX_4$ is regular.  Since
regularity is an open condition, again shrinking $X_4$ if
necessary, we may assume that $X_4$ is regular. Finally, we may
suppose (shrinking again if necessary) that $X_4$ is connected.
Whence the extension rank of $X_4$ is 32.

\vspace{.1in}

\smallskip

{\it Step 2: Thicken $X_4$ to a 5-dimensional integral manifold
$X_5$.} To apply the Cartan-K\"ahler theorem we must construct a
10-dimensional manifold $Z_4$ containing $X_4$ so that $T_zZ_4$
and $H(T_zX_4)$ are transverse along $X_4$. Define $W_5 \subset
M_7\bR$ to be the 5-dimensional subspace

\vspace{.1in}

\begin{displaymath}
  \left( \begin{array}{ccccccc}
    x_1 & 0 & 0 & x_2 & 0 & 0 & 0 \\
    0 & x_1 & 0 & x_2 & 0 & 0 & 0 \\
    0 & 0 & x_1 & x_2 & 0 & 0 & 0 \\
    0 & 0 & 0 & 0 & 0 & 0 & 0 \\
    x_5 & x_3 & x_3 & 0 & 0 & 0 & 0 \\
    x_4 & x_5 & x_3 & 0 & 0 & 0 & 0 \\
    x_4 & x_4 & x_5 & 0 & 0 & 0 & 0
  \end{array} \right) \, , \quad (x_1,x_2,x_3,x_4,x_5) \in \bR^5 \, .
\end{displaymath}

\vspace{.1in}

Notice that $W_1 \subset W_5$, $W_5 \cap \fh_4 = \{0\}$, and $RW_5
= W_5$. Since $\fg_2 \subset \fh_4$, the affine space $I_7 + W_5$
intersects $G_2$ transversely at $I_7 \in \tGL(7)$.  Hence there
is a neighborhood $U_5$ of $0 \in W_5$ such that the map $U_5 \to
\tGL(7)/\G2$ sending $x \to (I+x)G_2$ is an embedding.

\vspace{.1in}

Define a 10-dimensional manifold $Z_4 \subset S$ by
\begin{displaymath}
  Z_4 \ := \ \left\{ \rho \left( p , (y^4 , y^5 , 0 , 0 ) , I_7 + x \right)
  \ | \ p \in A \, , \ (y^4,y^5) \in \bR^2 \, , \ x \in U_5 \right\} \, .
\end{displaymath}

\vspace{.1in}

By construction $Z_4$ contains $X_4$, is $r$-invariant, and the
intersection $H(T_z X_4) \cap T_z Z_4$, $z\in X_4$, is of
dimension five.  Thus, the Cartan-K\"ahler theorem yields a
5-dimensional, real analytic integral manifold $Y^5$ such that
$X_4 \subset Y_5 \subset Z_4$.  Since $\cI$ is preserved under
$r$, $r(Y_5)$ is also an integral manifold.  The $r$-invariance of
$X_4$ and $Z_4$ implies $X_4 \subset r(Y_5) \subset Z_4$. The
uniqueness portion of the Cartan-K\"ahler theorem assures us that
$X_5 := Y_5 \cap r(Y_5)$ is also a 5-dimensional, real analytic
integral manifold of $\cI$.

\vspace{.1in}

Given $z \in X_3$, $T_z X_5 = H(T_zX_4) \cap T_zZ_4$ is (i)
$\overline\pi$-transverse, and (ii) the $\overline E_5$ of a
canonical regular flag.  Transversality implies that a
neighborhood of $X_3$ in $X_5$ projects diffeomorphically onto a
neighborhood $N_5$ of $A$ in $A \times \bR^2 \subset M$.  So,
shrinking $X_5$ if necessary, we may assume that it is the
(connected) image of a section $N_5 \to S$. Moreover, since
regularity is an open condition, a neighborhood of $X_3$ in $X_5$
will be regular.  Hence, again shrinking $X_5$ if necessary, we
may take $X_5$ to be a regular integral manifold.  The extension
rank of $X_5$ is 21.

\vspace{.1in}

\smallskip

{\it Steps 3 \& 4.}  As in Steps 1 \& 2 we may thicken $X_5$ to a 6-dimensional integral manifold $X_6$ of extension rank 7.  Then $X_6$ is thickened to a 7-dimensional integral manifold $X_7$ that is an $r$-invariant connected image of a section $\sigma : N \to S$ over an open
neighborhood $N$ of $A$ in $M$.

\vspace{.1in}

\smallskip

{\it The finish.}  As a section $N \to S$, $\sigma$ represents a
$G_2$-structure on the 7-dimensional $N$.  The corresponding
3-form on $N$ is $\varphi := \sigma^*(\widehat{\varphi_0})$.  By
construction $\sigma(N) \subset S$ is an integral manifold of
$\cI$.  Equivalently, the $G_2$-structure is torsion-free, and
$(N,\varphi)$ is a $G_2$-manifold.

\vspace{.1in}

The relation $r \circ \sigma = \sigma \circ r$ implies that
$r:M\to M$ restricts to an involution on $N$, and that $r^*
\varphi = \varphi$.  Whence $r$ is a $G_2$-involution
(\S\ref{sec:G2invol}), Lemma \ref{involution1}. It follows
immediately that $A$, as the fixed point locus of $r$ in $N$, is
associative.

\vspace{.1in}

Let $h$ denote the metric induced on $N$ by the $G_2$-structure.
At $z \in A$, $\{ \w_1 , \, \w_2 , \, \w_3 , \, \dy^4 , \, \dy^5 ,
\, \dy^6 , \, \dy^7 \}$ is a $G_2$ coframing of $T^*_zN$.  In
particular,
\begin{eqnarray*}
  \varphi_z & = &
  \w_{123} + \w_1 \wedge \left( \dy^{45} + \dy^{67} \right)
          + \w_2 \wedge \left( \dy^{46} - \dy^{57} \right) \\ & &
          + \w_3 \wedge \left( - \dy^{47} - \dy^{56} \right) \, , \\
  h_z & = & \w_1{}^2 + \w_2{}^2 + \w_3{}^2 + (\dy^4)^2 + (\dy^5)^2
          + (\dy^6)^2  + (\dy^7)^2 \,  .
\end{eqnarray*}
Whence $h_{|A} = g$, and the inclusion $i : A \hookrightarrow N$
is an isometry.  (Also, $i^*\varphi = \w_{123} = \tvol_A$, proving
again that $A$ is associative.)  This completes the proof of Theorem \ref{thm:assoc}.

\medskip

\noindent{\it Remark.}  We remarked in \S\ref{sec:intro} after the statement of Theorem \ref{thm:assoc} that, so long as $A$ is not flat, $\tHol(N) = G_2$.  In particular, $N \not= CY \times \bR$.  This may be seen as follows.  Suppose that $\tHol(N) \not= G_2$.  Then $\tHol(N) \subseteq \tSU(3)$ and $N$ is an $\tSU(3)$-manifold.

Set $\bi = \sqrt{-1}$ and take the $\tSU(3)$ action on $\bR^7 = \bC^3 \op \bR$ that fixes the forms $\td x^7$,  
\begin{eqnarray*}
  \w_0 & = & \td x^{16} - \td x^{25} - \td x^{34} \, , \\
  \Upsilon_0 & = & ( \td x^1 + \bi \, \td x^6 ) \wedge 
                   ( \td x^2 - \bi \, \td x^5 ) \wedge
                   ( \td x^3 - \bi \, \td x^4 ) \, .
\end{eqnarray*}
Then $\tSU(3)$ acts trivially on the second factor and by the standard representation on the first.

Let $R = \cF/\tSU(3)$ and consider 
\begin{center}
\setlength{\unitlength}{1cm}
\begin{picture}(2,3)
  \put(0.3,2.7){$\cF$}
  \put(0.4,2.5){\vector(0,-1){2}}
  \put(0.2,0){$N$}
  \put(0.8,2.8){\vector(2,-1){1}}
  \put(1.9,0.75){$S$}
  \put(2.0,1.8){\vector(0,-1){0.65}}
  \put(1.9,2.0){$R$}
  \put(1.7,0.7){\vector(-2,-1){1}}
  \put(0,1.5){$\pi$}
  \put(1.3,0.15){$\overline \pi$}
  \put(1.4,2.7){$\mu$}
  \put(2.1,1.5){$\nu$}
\end{picture}
\end{center}
Above, $\nu \circ \mu = \rho$.  Let $\tilde \pi = \overline \pi \circ \nu : R \to N$.

If $\tHol(N) \subseteq \tSU(3)$, then $N$ is an $\tSU(3)$-manifold.  In particular, $N$ admits a section $\tau : N \to R$ such that $\nu \circ \tau = \sigma$.  Rename $\cI = \cI_{G_2}$, and let $\cI_{\tSU(3)}$ be the ideal generated by $\td \eta^7$, $\td \w$ and $\td \Upsilon$, where  
\begin{eqnarray*}
  \w & = & \eta^{16} - \eta^{25} - \eta^{34} \\
  \Upsilon & = & ( \eta^1 + \bi \, \eta^6 ) \wedge 
                   ( \eta^2 - \bi \, \eta^5 ) \wedge
                   ( \eta^3 - \bi \, \eta^4 ) \, .
\end{eqnarray*}
The image $\tau(N)$ is necessarily an integral manifold of $\cI_{\tSU(3)}$.

Fix $z \in X_3$, and let $E \in V_7(\cI_{G_2}, \pi)$ denote the 7-plane constructed in the proof above with $\rho_*E = T_z X_7 = T_z \sigma(N) \in V_7(\cI_{G_2} , \overline \pi)$.  If $\tau : N \to R$ exists, then there exists $E' \in V_7( \cI_{\tSU(3)} , \pi ) \subset V_7(\cI_{G_2} , \pi)$ such that $\mu_* E' \in V_7(\cI_{\tSU(3)} , \tilde\pi)$ is tangent to $\tau(N)$ and $\rho_* E' = \nu_* ( \mu_* E' ) = \rho_* E$.  As noted in \S\ref{sec:int_elem}, this implies $E' \equiv E \tmod \fg_2$.  A lengthy computation confirms that this is possible if and only if $A$ is flat.

%-------------------------------------------------------------------------
\section{The Cayley embedding: proof of Theorems \ref{thm:cayley} \& \ref{thm:opencayley}}
\label{sec:cayleyproof}
%-------------------------------------------------------------------------
\noindent

{\it Remark.} As an application of the \CK Theorem the proof of
Theorems \ref{thm:cayley} and \ref{thm:opencayley} is very like the proof of Theorems \ref{thm:assoc} and \ref{thm:openassoc}.  However, unlike 3-manifolds, the 4-manifold $A=\mathit{int}(K)$
may not admit a global parallelism.  We assume that the bundle of
self-dual 2-forms on $\mathit{doub}(K)$ is trivial in order to obtain the structure necessary to apply the \CK Theorem.

\vspace{.1in}

The group $\tSpin(7)$ is admissible \cite{BrExHol};
$\Lambda^*(\bR^8)^{\tSpin(7)}$ is generated by the 4-form
$\Psi_0$.
The differential system $\cI$ of \S\ref{sec:ideal} is generated by the
exterior derivative of
\begin{eqnarray*}
  \widehat{\Psi_0} & = &
  \eta^{0123} + \eta^{4567} +
  \left( \eta^{01} + \eta^{23} \right) \wedge
  \left( \eta^{45} +\eta^{67} \right) \\ & &
  + \left( \eta^{02} + \eta^{31} \right) \wedge
  \left( \eta^{46} + \eta^{75} \right)
  + \left( \eta^{03} + \eta^{12} \right) \wedge
  \left( \eta^{74} + \eta^{65} \right) \, .
\end{eqnarray*}
The condition that $\td \widehat{\Psi_0}=0$ is 56 independent linear equations on the
functions $p^j_{k\ell}(u)$ parameterizing $\tGr_8(T_u \cF, \pi)$ (cf. \S\ref{sec:int_elem}).  Thus
\begin{displaymath}
  \tcodim ( V_8(\cI,\pi) , \tGr_8(T\cF) ) \ = \ 56 \, .
\end{displaymath}
Since $\tdim \, \fh^0(\mathfrak{spin}(7)) = 56$
\cite[Prop.4]{BrExHol}, it follows that $\tSpin(7)$ is strongly
admissible, c.f. \S\ref{sec:str_adm}.

\vspace{.1in}

Fix $E \in V_8(I,\pi)$, and let $F$ denote the canonical flag of
\S\ref{sec:flags}. The subspaces $\mathfrak{h}_k$ of
\S\ref{sec:flags} are: $\mathfrak{h}_0 = \mathfrak{h}_1 =
\mathfrak{h}_2 = \fh_3 = M_8\bR \simeq \bR^{64}$. The subspace
$\fh_4 \subset M_8\bR$ is defined by the single equation
$x^0_0+x^1_1+x^2_2+x^3_3 = 0$.  As in \S\ref{sec:assocproof} calculations of the remaining  $\fh_j$ lead us to the polar space codimensions: $(c_0, c_1, \ldots, c_8) = (0,0,0,0,1,5,15,35,43)$ and $\sum c_j = 56$, so that $\tSpin(7)$ is regularly presented, c.f \S\ref{sec:flags}.  Cartan's Test concludes that $V_8(\cI,\pi)$ is a codimension 56 (observed above) submanifold of $\tGr_8(T\cF, \pi)$, and each $E \in V_8(\cI,\pi)$ is the terminus of a canonical regular flag $F$.
\medskip

%\hrule

%\begin{center}{\it Issue of whether or not double(A) admits trivial bundle needs to be straightened out.}\end{center}\hrule\smallskip

\vspace{.1in}

This completes the necessary preliminaries for $\tSpin(7)$. For
Theorem \ref{thm:cayley} we assume that $(K^4,g)$ is a closed
oriented, real analytic Riemannian 4-manifold, and that the bundle
$\Lambda^2_+(K)$ of self-dual 2-forms over $K$ is smoothly
trivial.  Bochner's result or Theorem \ref{thm:altboch} implies
that $\Lambda^2_+(K)$ is real-analytically trivial.  In
particular, there exist globally defined real-analytic self-dual
2-forms $\Omega_1, \, \Omega_2 , \, \Omega_3$ such that $\Omega_j
\wedge \Omega_k = 2 \delta_{jk} \, \tvol_g$.

\vspace{.1in}

In the case of Theorem \ref{thm:opencayley}, invoke Lemma
\ref{lem:analyticstr} to endow $\mathit{doub}(K)$ with a
compatible real analytic Riemannian structure
$(\mathit{doub}(K),g)$.  (If $K$ is closed, then $\mathit{doub}(K)
= K$.)  As above, $\mathit{doub}(K)$ admits globally defined
real-analytic self-dual 2-forms $\Omega_1, \, \Omega_2 , \,
\Omega_3$ such that $\Omega_j \wedge \Omega_k = 2 \delta_{jk} \,
\tvol_g$.

\vspace{.1in}

The rest of the argument applies to both theorems.  Let $A =
\mathit{int}(K)$.  (In Theorem \ref{thm:cayley}, $A = K$.)  Assume $A$ is connected, else apply the theorem to each connected component individually.
Restrict the Riemannian metric $g$ and 2-forms $\Omega_j$ to $A$.
While $A$ may not admit a global coframing, it is not difficult to
check that there exist local orthonormal coframings $\{ \w^a
\}_{a=1}^4$ such that
\begin{displaymath}
  \Omega_1 = \w^1 \wedge \w^2 + \w^3 \wedge \w^4 \, , \quad
  \Omega_2 = \w^1 \wedge \w^3 - \w^2 \wedge \w^4 \, , \quad
  \Omega_3 = \w^1 \wedge \w^4 + \w^2 \wedge \w^3 \, .
\end{displaymath}

This choice of coframing is unique up to the action of
$\textrm{SU}(2)$.  (Recall, $\mathrm{SU}(2)$ is the subgroup of
$\mathrm{SO}(4)$ preserving the two forms $\frac{i}{2} ( \zeta^1
\wedge \overline \zeta^1 + \zeta^2 \wedge \overline \zeta^2 )$ and
$\zeta^1 \wedge \zeta^2$, where $\zeta^1 = \w^1 + i \w^2$ and
$\zeta^2 = \w^3 + i \w^4$.)  We may identify $\tSU(2)$ with the
subgroup of $\tSpin(7)$ fixing $\dx^j$, $j=0 , \ldots , 3$, via
$$
  \left\{ \left( \left. \begin{array}{cc} \tId_4 & 0 \\ 0 & P \end{array} \right) \right| P \in \tSU(2) \right\} \subset \tSpin(7) \, .
$$

Let $M = \bR^4 \times A$ with linear coordinates $\{ y^j \}_{j=0}^3$ on $\bR^4$.
Then
$$
  \Psi \ = \ \dy^{0123} + \half \Omega_1 \wedge \Omega_1 +
  \left( \dy^{01} + \dy^{23} \right) \wedge \Omega_1
  + \left( \dy^{02} + \dy^{31} \right) \wedge \Omega_2
  - \left( \dy^{03} + \dy^{12} \right) \wedge \Omega_3 \,
$$
defines a $\tSpin(7)$-structure on $M$.  Let $\sigma : M \to S$ denote the corresponding section.

\vspace{.1in}

Define an involution on $M$ by $r(y,p) = (-y,p)$.  Note that
$r^*\Psi = \Psi$.  Define a covering involution $r : \cF \to \cF$
as follows.  Given $u : T_zM \to \bR^8$, let $r(u)$ be the coframe
$r^*(u) : T_{r(z)} M \to \bR^8$.  Then $r^*(\eta) = \eta$ and
$r^*\widehat{\Psi_0} = \widehat{\Psi_0}$. Let $\{ \w^a \}$ be a
coframing of an open set $U \subset L$.  Then $\{ \td y^j , \w^a
\}$ defines a trivialization $\cF_{|\bR^4\times U} :=
\pi\inv(\bR^4\times U) \simeq \bR^4 \times U \times \tGL_8\bR$.
Notice that $\cF_{|\bR^4\times U}$ is invariant under $r$ and,
with respect to the trivialization, is given by $r(y,p,g) =
(-y,p,Rg)$, where
\begin{displaymath}
  R \ = \ \left( \begin{array}{cc} -I_4 & 0 \\ 0 & I_4 \end{array} \right) \in \tSpin(7) \, .
\end{displaymath}
In particular, $r$ descends to a well-defined involution on $S$.

\vspace{.1in}

From this point on the proof of Theorem \ref{thm:cayley} is very
similar to the proof of Theorem \ref{thm:assoc}; so we merely
sketch the main steps.  Define $X_4 = \sigma(A)$.  Since $\cI$ is
generated by a 5-form, and $X_4$ is 4-dimensional, $X_4$ is
trivially an integral manifold of $\cI$.  Since $r \circ \sigma =
\sigma \circ r$, it follows that $X_4$ lies in the fixed point
locus of $r : S \to S$.

\vspace{.1in}

It remains to select the subspaces $W_{d_1} \subset W_{d_2}
\subset W_{d_3} \subset W_{d_4} \subset M_8\bR$,
$(d_4,d_5,d_6,d_7) = (1,5,15,35)$, so that (i) $\tdim W_{d_j} =
d_j$, (ii) $W_{d_s} \cap \fh_s = \{0\}$, and (iii) $R W_{d}
\subset W_{d}$.  Because the coframing $\w, \dy$ of $M$ is defined
only up to the $\tSU(2)$ action it is also necessary that we pick
the subspaces so that $\tSU(2) \,  W_d \subset W_d$.  We leave this exercise to the reader.

%-----------------------------------------------------------------------


\begin{thebibliography}{[FP]}
%-----------------------------------------------------------------------

\bibitem{AK1} Akbulut, S. and King, H.,  {\it The Topology of Real Algebraic Sets with Isolated
Singularities}, Annals of Math. {\bf 113}, (1981), 425--446.


\bibitem{AK2} Akbulut, S. and King, H.,  {\it A Relative Nash Theorem},  Trans. of AMS, Vol.{\bf 267}, No.{\bf 2}, (1981),
465--481.

\bibitem{AK3} Akbulut, S. and King, H.,  {\it Topology of Real
Algebraic Sets}, MSRI Book Series, No.{\bf 25}, Springer-Verlag,
1992.

%\bibitem{AkSa} Akbulut, S. and Salur, S., {\it Calibrated Manifolds and Gauge Theory},
%math.GT/0402368, 2004.

\bibitem{Be} Besse, Arthur L.,
{\it Einstein manifolds},
Ergebnisse der Mathematik und ihrer Grenzgebiete (3)  {\bf 10}, Springer-Verlag,
Berlin, 1987.

\bibitem{Bochner} Bochner, S.,  {\it Analytic Mapping of Compact {R}iemann Spaces into {E}uclidean Space}, Duke Math. J., Vol.{\bf 3},
No.{\bf 2}, (1937), 339--354.


\bibitem{Boothby} Boothby, W. M., {\it An Introduction to Differentiable Manifolds and Riemannian Geometry}, Pure and Applied Mathematics,
    Vol. {\bf 120}, Academic Press Inc., Second Edition, 1986.

\bibitem{BrExHol} Bryant, R. L., {\it Metrics with Exceptional Holonomy}, Annals of Math. {\bf 126}, (1987),
525--576.

\bibitem{BrClbEm} Bryant, R. L., {\it Calibrated Embeddings in the Special {L}agrangian and Coassociative cases},
Ann. Global Anal. Geom., Special issue in memory of Alfred Gray
(1939--1998), Vol.{\bf 18}, No.{\bf 3-4},(2000), 405--435, math.DG/9912246.

\bibitem{BCG3} Bryant, R. L., Chern, S. S., Gardner, R. B., Goldschmidt, H. L. and Griffiths, P. A.,
{\it Exterior Differential Systems}, Mathematical Sciences
Research Institute Publications, Vol.{\bf 18}, Springer-Verlag,
1991.

\bibitem{DK} DeTurck, Dennis M. and Kazdan, Jerry L.,
{\it Some regularity theorems in Riemannian geometry},
Ann. Sci. cole Norm. Sup. (4) {\bf 14} (1981), no. 3, 249--260.

\bibitem{FG} Fern{\'a}ndez, M. and Gray, A., {\it Riemannian Manifolds with Structure Group {$G\sb{2}$}}, Ann. Mat. Pura Appl. (4),
Vol.{\bf 132}, (1982), 19--45 (1983),

\bibitem{HL} Harvey, F.R. and Lawson, H.B., {\it Calibrated Geometries}, Acta. Math. {\bf 148} (1982), 47--157.

\bibitem{IL} Ivey, T. A. and Landsberg, J. M., {\it Cartan for Beginners: Differential Geometry via Moving Frames
              and exterior differential systems}, Graduate Studies in Mathematics, Vol. {\bf 61}, American Mathematical Society, 2003.

\bibitem{Joy1} Joyce, D.D., {\it Compact Manifolds with Special Holonomy}, OUP, Oxford, 2000.

\bibitem{KN1} Kobayashi, S. and Nomizu, K., {\it Foundations of Differential Geometry. {V}ol. {I}}, Wiley Classics Library, Reprint of the 1963 original, John Wiley \& Sons
              Inc., 1996.
\bibitem{Kr1} Krantz, S.G. and Parks, H.R., {\it A Primer of Real
Analytic Functions}, Birkh\"{a}user Verlag, Basel, 1992.


\bibitem{Kr2} Krantz, S.G. and Parks, H.R., {\it The Implicit
Function Theorem. History, Theory and Applications},
Birkh\"{a}user Verlag, Boston, 2002.


\bibitem{McLe} McLean, R.C., {\it Deformations of Calibrated Submanifolds}, Comm. Anal. Geom. {\bf 6} (1998), 705--747.

\bibitem{MilSta} Milnor, J. W. and Stasheff, J. D., {\it Characteristic Classes}, Annals of Mathematics Studies, No. 76,
Princeton University Press, 1974.

\bibitem{N} Nash, J., {\it Real Algebraic Manifolds}, Annals of Math. {\bf 56}, (1952), 405--421.

\bibitem{T} Tognoli, A., {\it Su una congettura di Nash}, Ann.
Scuola Norm. Sup. Pisa. Sci. Fis. Math. {\bf 27}, (1973),
167--185.



\end{thebibliography}
\end{document}